\documentclass[reqno,11pt]{amsart}
\usepackage[a4paper,margin=3truecm]{geometry}
\usepackage{amssymb,amsfonts,amsbsy,graphicx,setspace,color,xfrac}
\usepackage[T1]{fontenc}
\usepackage{calligra}
\DeclareMathAlphabet{\mathcalligra}{T1}{calligra}{m}{n}

\usepackage[usenames,dvipsnames]{pstricks}
\usepackage{IEEEtrantools}

\makeatletter

\everymath{\displaystyle} 

 \theoremstyle{plain}
 \newtheorem{thm}{Theorem}[section]
 \newtheorem{lmm}[thm]{Lemma}
 
 \numberwithin{equation}{section} 
 \numberwithin{figure}{section} 
 \newtheorem{prop}[thm]{Proposition}
 \theoremstyle{remark}
 \newtheorem{rmk}[thm]{Remark}
 
 \newtheorem*{acknowledgement*}{Acknowledgement}
 \theoremstyle{definition}
 \newtheorem{defi}[thm]{Definition}

\newcommand{\gra}[1]{\left \{ #1 \right \}}
\newcommand{\be}{\begin{equation}}
\newcommand{\ee}{\end{equation}}
\newcommand{\abs}[1]{\left| #1 \right|}
\newcommand{\norm}[1]{\lVert #1 \rVert}

\newcommand{\A}{\mathcal A}
\newcommand{\Sy}{\mathcal S}
\newcommand{\F}{\mathcal F}
\newcommand{\Pb}{\mathcal P}
\newcommand{\R}{\mathbb R}
\newcommand{\Z}{\mathbb Z}
\newcommand{\N}{\mathbb N}

\newcommand{\Co}{{\bf\mathfrak{C}}}
\newcommand{\eps}{\varepsilon}
\newcommand{\wconv}{\rightharpoonup}
\newcommand{\rnd}{\R^{Nd}}

\newcommand{\bpsi}{\boldsymbol{\psi}}
\newcommand{\smallspace}{\hspace{0.25cm}}



\date{\today}
\author{Ugo Bindini}

\author{Luigi De Pascale}
\title[Semiclassical Limit of DFT]{Optimal transport with Coulomb cost and the Semiclassical Limit of Density Functional Theory}

\begin{document}

\begin{abstract} 
We present some progress in the direction of determining the semiclassical limit of the Hoenberg-Kohn universal functional in Density Functional Theory for Coulomb systems.
In particular we give a proof of the fact that for Bosonic systems with an arbitrary number of particles the limit is the multimarginal optimal transport problem with Coulomb cost and that the same holds for Fermionic systems with 2 or 3 particles.
Comparisons with previous results are reported .
The approach is based on some techniques from the optimal transportation theory.
\end{abstract} 

\maketitle

\section{Introduction and preliminary results}
The ground state of a system of $N$ electrons interacting Coulombically between each others and with $M$ nuclei is described by one of the following minimal values (ground values)  
\begin{equation}\label{bosgrst}
E^\Sy_\hbar = \min_\Sy \  T_\hbar (\psi) + V_{ee}(\psi)- V_{ne} (\psi), 
\end{equation}
or
\begin{equation}\label{fergrst}
E^\A_\hbar = \min_\A \  T_\hbar (\psi) + V_{ee}(\psi)- V_{ne} (\psi), 
\end{equation}
and by the corresponding minimizers (ground states).

The minimization domains are the following sets of wave functions  
\begin{multline*} 
\Sy=\{\psi \in H^1((\R^d \times \Z_2)^N) \ : \ \int |\psi|^2 dz =1 , \ \mbox{for all permutations} \ \sigma \ \mbox{of}\ N\ \mbox{points}\\ 
  \psi((x_{\sigma(1)},\alpha_{\sigma(1)}), \dots, (x_{\sigma(N)}, \alpha_{\sigma(N)}))=\psi((x_1,\alpha_1), \dots, (x_N, \alpha_N)) \} , 
 \end{multline*}
\begin{multline*}
\A=\{\psi \in H^1((\R^d \times \Z_2)^N) \ : \ \int |\psi|^2 dz =1 , \ \mbox{for all permutations} \ \sigma \ \mbox{of}\ N\ \mbox{points}\\ 
  \psi((x_{\sigma(1)},\alpha_{\sigma(1)}), \dots, (x_{\sigma(N)}, \alpha_{\sigma(N)}))=sign(\sigma) \psi((x_1,\alpha_1), \dots, (x_N, \alpha_N)) \} , 
 \end{multline*}
where we adopted the common notation
 $$ \int f(z) dz:= \int f(z_1, \dots, z_N) dz_1\dots dz_N:= \sum_{\alpha_1, \dots, \alpha_N=0,1} \int_{\R^d} f(x_1,\alpha_1; \dots: x_N, \alpha_N) dx_1\dots dx_N.$$
The three terms in the functionals are: 
the kinetic energy 
$$T_\hbar(\psi)= \frac{\hbar^2}{2m} \int | \nabla \psi |^2 (z) dz,$$
the electron-electron interaction energy
$$V_{ee}(\psi) = \int  \sum_{1\leq i < j \leq N} \frac{|\psi|^2 (z)}{|x_i-x_j|} dz$$
and the nuclei-electrons interaction energy
$$V_{ne}(\psi) = \int  \sum_{1\leq i < j \leq N} \sum_{k=1}^M\frac{|\psi|^2 (z)}{|x_i-N_k|} dz.$$

The values $E^\Sy_\hbar$ and $E^\A_\hbar $ represents, respectively, the ground value for a Bosonic and Fermionic system of particles 
with $N$ electrons and $M$ nuclei. The $N_k$ are the positions of the nuclei, while the couple
$(x_i,\alpha_i)$ is the position-spin of the $i$-th electron. In the usual Born interpretation, $|\psi|^2$ is the probability distribution of the $N$ electrons and, according with the indistinguishability principle, it is invariant with respect to permutations of the $N$ $(x_i,\alpha_i)$ variables.

Computing the ground values above amounts to solve a Schr\"odinger equation in $\R^{dN}$ and the numerical cost scale exponentially with $N$.
 The Density Functional Theory (DFT from now on) is an alternative introduced in the late sixties 
by Hohemberg, Kohn and Sham. However the desire to describe the system in term of a different variable is much older and we may consider 
the Thomas-Fermi model as a precursor of this theory.

To every wave function $\psi$ it is associate a probability density on $\R^d$ defined as follows \footnote{In the usual definition the integral in the definition of $\rho_\psi$ is also multiplied by a factor $N$, but here we prefer to deal with probability measures.}:
$$ \rho_\psi(x):= \sum_{\alpha_1, \dots, \alpha_N=0,1} \int_{\R^{d(N-1)}} |\psi |^2 (x,\alpha_1; \dots: x_N, \alpha_N) dx_2\dots dx_N.$$
The map which associate $\rho$ to $\psi$ will be denoted by $\psi \downarrow \rho_\psi$ and $\rho$ will be called single electron density or electronic density. 
If we wish to describe the system in terms of $\rho$, 
the problems above  should be reformulated as a minimization of suitable energies with respect to $\rho$ which, no matter how big is $N$, 
is always a probability measure in $\R^d$.
The exact image of the map $\psi \downarrow \rho_\psi$ was characterized by Lieb \cite{lieb1983} who proved that the image of both $\Sy$ and $\A$ is 
$$ H=\{\rho \ | \ 0 \leq \rho, \ \int \rho =1, \ \sqrt{\rho} \in H^1 (\R^d)  \}.$$

For $\rho \in H$  we introduce the Hohenberg-Kohn functional\footnote{Usually these functionals are denoted by $F_{HK} (\psi)$. For the sake of a lighter notation, here we will not report 
the $HK$.}
\begin{equation}\label{HK1}
F^\Sy_\hbar (\rho) = \inf_{\psi \in \Sy, \ \psi \downarrow \rho} \{ T_\hbar(\psi) +V_{ee}(\psi)\}, 
\end{equation}
\begin{equation}\label{HK2}
F^\A_\hbar (\rho) = \inf_{\psi \in \A, \ \psi \downarrow \rho} \{ T_\hbar(\psi) +V_{ee}(\psi)\}.
 \end{equation}
Then the problems above can be reformulated as follows:
$$E^\Sy_\hbar = \min_{\rho \in H }\  F^\Sy_\hbar (\rho) + N \int v(x) \rho(x) dx , $$
and 
$$E^\A_\hbar = \min_{\rho \in H }\  F^\A_\hbar (\rho) + N \int v(x) \rho(x) dx ,$$ 
where  we denoted 
$$ v(x)= \sum_{k=1}^M \frac{1}{|x-N_k|}.$$

This paper concerns the semiclassical limit of the two functionals $F^\Sy_\hbar (\rho)$ and $F^\A_\hbar (\rho)$ and more precisely it concerns the relations of the multimarginal optimal transport theory 
with these semiclassical limits.
The ties between the DFT for Coulomb systems and Optimal transport appeared first in \cite{buttazzo2012, cotar2013density} and by now have revealed to be a precious tool 
for the understanding. To detail better the results, let us shortly introduce the multimarginal optimal transport problem of interest here.

For every $\rho \in H$ consider the set 
$$\Pi (\rho):= \{P \in \Pb (\R^{dN}) \ | \ \pi_\sharp ^i P= \rho , \ i=1, \dots, N\}, $$
and then the multimarginal optimal transport problem with Coulomb cost
\begin{equation}\label{coulomb}
C(\rho):= \min_{P \in \Pi(\rho)}  \int  \sum_{1\leq i < j \leq N} \frac{1}{|x_i-x_j|} dP.
\end{equation}
Duality for (\ref{coulomb}) and some properties of the functional $C(\rho)$ have been studied in \cite{depascale2015,buttazzo2016}. The structure of $1-$dimensional minimizers has been investigated in  \cite{colombo2013multimarginal}. 

The results on optimal transportation needed in this paper will be reported in the next section.

In this paper we will prove that  
\begin{thm}\label{bosonconv}  For all $\rho \in H$ and $d,N \in \N$, 
$$F^\Sy_\hbar (\rho)  \stackrel{\hbar \to 0}{\rightarrow}  C(\rho),$$
\end{thm}
and that
\begin{thm}\label{fermionconv}  For all $\rho \in H$, $d=1,2,3,4$, $N=2,3$
$$ F^\A _\hbar (\rho)   \stackrel{\hbar \to 0}{\rightarrow}  C(\rho).$$
\end{thm}
The same problem was previously studied in \cite{cotar2013density}. It is easy to see that the proof presented in that paper adapt to prove Theorem \ref{bosonconv} for any $N$. In this case our contribution consists in a more direct use of optimal transport techniques with the consequent simplifications.
The convergence \ref{fermionconv}  was proved in  \cite{cotar2013density} for $N=2$. Here we extend the result to $N=3$ and we are able to enlarge the class of approximating antisymmetric wave functions also for $N=2$.
\begin{rmk}
Although the usual description is limited  to the physical dimension $d=3$, here we explored also other dimensions in the hope to shade some light on the problems which are still open. 
\end{rmk}

Since the functionals appearing in Theorems \ref{bosonconv} and \ref{fermionconv} above are all expressed as minimal values, the natural tool to deal with their convergence is the $\Gamma$-convergence which also we shortly introduce in the next section.

\section{Tools: $\Gamma$-convergence and Multimarginal Optimal Transport}

\subsection{Definition of $\Gamma$-convergence and basic results}

A crucial tool that we will use throughout this paper is $\Gamma$-convergence. 
All the details can be found, for instance, in Braides's book \cite{braides2002gamma} or in the classical book by 
Dal Maso \cite{dal1993introduction}. In what follows, $(X,d)$ is a metric space or a topological space equipped with a convergence.

\begin{defi} Let $(F_n)_n$ be a sequence of functions $X \mapsto \bar\R$. We say that $(F_n)_n$ $\Gamma$-converges to 
$F$ and we write $F_n \xrightarrow[n]{\Gamma} F$ if for any $x \in X$ we have
\begin{itemize}
\item for any sequence $(x_n)_n$ of $X$ converging to $x$
$$ \liminf\limits_n F_n(x_n) \geq F(x) \qquad \text{($\Gamma$-liminf inequality);}$$
\item there exists a sequence $(x_n)_n$ converging to $x$ and such that
$$ \limsup\limits_n F_n(x_n) \leq F(x) \qquad \text{($\Gamma$-limsup inequality).} $$
\end{itemize} \end{defi}

This definition is actually equivalent to the following equalities for any $x \in X$:
$$ F(x) = \inf\left\{ \liminf\limits_n F_n(x_n) : x_n \to x \right\} = \inf\left\{ \limsup\limits_n F_n(x_n) : x_n \to x \right\} $$
The function $x \mapsto  \inf\left\{ \liminf\limits_n F_n(x_n) : x_n \to x \right\}$ is called $\Gamma$-liminf of the sequence $(F_n)_n$ and the other one its $\Gamma$-limsup. A useful result is the following (which for instance implies that a constant sequence of functions does not $\Gamma$-converge to itself in general).

\begin{prop}\label{lsc} The $\Gamma$-liminf and the $\Gamma$-limsup of a sequence of functions $(F_n)_n$ are both lower semi-continuous on $X$. \end{prop}

The main interest of $\Gamma$-convergence resides in its consequences in terms of convergence of minima:

\begin{thm} \label{convminima} Let $(F_n)_n$ be a sequence of functions $X \to \bar\R$ and assume that $F_n \xrightarrow[n]{\Gamma} F$. Assume moreover that there exists a compact and non-empty subset $K$ of $X$ such that
$$ \forall n\in N, \; \inf_X F_n = \inf_K F_n $$
(we say that $(F_n)_n$ is equi-mildly coercive on $X$). Then $F$ admits a minimum on $X$ and the sequence $(\inf_X F_n)_n$ converges to $\min F$. Moreover, if $(x_n)_n$ is a sequence of $X$ such that
$$ \lim_n F_n(x_n) = \lim_n (\inf_X F_n)  $$
and if $(x_{\phi(n)})_n$ is a subsequence of $(x_n)_n$ having a limit $x$, then $ F(x) = \inf_X F $. 
\end{thm}

\subsection{Multimarginal optimal transportation and composition of optimal transport plans}

In this subsection we present some basic results about multimarginal optimal transportation with Coulomb cost. 

An element $P$ of $\Pi(\rho)$ is commonly called a transport plan for $\rho$.

If $\rho \in \mathcal{P}(\R^d)$, define
\[
 \mu_{\rho}(t) = \sup_{x \in \R^d} \rho(B(x,t)).
\]

\begin{defi} The \emph{concentration} of $\rho$ is
\[
 \mu_{\rho} = \lim_{t \to 0} \mu_{\rho}(t).
\]
\end{defi}

Note that $\rho \in H$ implies $\mu_\rho = 0$, since $\sqrt{\rho} \in H^1$ implies $\rho \in L^1$.

It is commonly assumed that, if $\mu_\rho < 1/N$, then $\mathcal{C}(\rho) < +\infty$. Here we provide a simple argument adapted to the particular case $\rho \in H$:

\begin{prop} \label{L1-L3} If $p \geq \frac{d}{d-1}$ and $\rho \in L^1(\R^d) \cap L^p(\R^d)$, then $\mathcal{C}(\rho) < +\infty$. 
\end{prop}

\begin{lmm} \label{lemma-L1-L3} Let $p \geq \frac{d}{d-1}$. Then there exists a constant $M$ such that, for every $\varphi \in L^1(\R^d) \cap L^p(\R^d)$,
\[
 \sup_{x \in \R^d} \int \frac{\varphi(y)}{\abs{x-y}} dy \leq M \norm{\varphi}_{L^1\cap L^p}. 
\]
\end{lmm}

\begin{proof} Fix $x \in \R^d$, and consider a parameter $a > 0$ to split
\[
 \frac{1}{\abs{x-y}} = \frac{\chi_{B(x,a)}(y)}{\abs{x-y}} + \frac{\chi_{B(x,a)^c}(y)}{\abs{x-y}} = f_a(y) + g_a(y).
\]

By H\"older's inequality,
\[
 \int \frac{\varphi(y)}{\abs{x-y}} dy \leq \norm{\varphi}_{L^p} \norm{f_a}_{L^{p'}} + \norm{\varphi}_{L^1} \norm{g_a}_{L^\infty},
\]
and it is easy to compute
\[
 \norm{f_a}_{L^{p'}} = \frac{\omega_d d}{d-p'}a^{d-p'} \ \ \ \  \mbox{and} \ \ \ \  \norm{g_a}_{L^\infty} = a,
\]
which gives the thesis and also allows to find the optimal value for the parameter $a$.
\end{proof}

We may now prove Proposition \ref{L1-L3}:

\begin{proof} Consider the transport plan $P(x_1,\dotsc, x_N) = \rho(x_1) \dotsm \rho(x_N)$; using Lemma \ref{lemma-L1-L3},
\begin{align*}
 C(P) &= \int c(X) dP(X) = \sum_{1 \leq i < j \leq N} \int_{\R^d \times \R^d} \frac{\rho(x_i)\rho(x_j)}{\abs{x_i-x_j}} dx_i dx_j \\
 {} &\leq \sum_{1 \leq i < j \leq N} \int_{\R^d} \gra{\sup_{x_j \in \R^d} \int_{\R^d} \frac{\rho(x_i)}{\abs{x_i-x_j}} dx_i} \rho(x_j) dx_j \\
 {} &\leq \sum_{1 \leq i < j \leq N} M \norm{\rho}_{L^1 \cap L^p} \int_{\R^d} \rho(x_j) dx_j \leq \sum_{1 \leq i < j \leq N} M \norm{\rho}_{L^1 \cap L^p},
\end{align*}
which is finite.
\end{proof}

\begin{defi}
 If $\alpha > 0$, let
 \[
  D_{\alpha} = \gra{X \in \rnd\ \mbox{s.t.}\ \exists i \neq j \text{ with } \abs{x_i-x_j} < \alpha}
 \]
 be an open strip around
 \[
  D_0 = \gra{X \in \rnd\ \mbox{s.t.}\ \exists i \neq j \text{ with } x_i = x_j},
 \]
 which is the set where the cost function $c$ is singular.
\end{defi}

\begin{defi}
 Given $P \in \Pi(\rho)$, we say that $P$ is \emph{off-diagonal} if
  \begin{equation} \label{P-diag} 
  P(D_\alpha) = 0 \ \  \text{for some $\alpha > 0$.}
  \end{equation}
\end{defi}

Note that if a transport plan $P$ has the property 
\eqref{P-diag}, then $C(P)$ is finite. The viceversa is not true in general, i.e. there may well be transport plans of finite cost whose supports have distance 0 from the set $D_0$.

However, this is not the case if $P$ is optimal, as recently proved in \cite{buttazzo2016}:

\begin{thm} \label{de-pascale}
 Let $\rho \in \Pb(\R^d)$ with $\mu_\rho < \tfrac{1}{N(N-1)^2}$, with $\mathcal{C}(\rho) < +\infty$, and let $\beta$ be such that
 \[
  \mu_{\rho}(\beta) < \frac{1}{N(N-1)^2}.
 \]

 If $P \in \Pi(\rho)$ is an optimal plan for the problem \eqref{coulomb}, then $P \mid _{D_\alpha} = 0$ for every $\alpha$ such that
 \[
  \alpha < \frac{2\beta}{N^2(N-1)}.
 \]
\end{thm}

\section{ A $\Gamma$-convergence result and proofs of the main theorems}
It is useful to reformulate the different variational problems so that they have a common domain. 
For every $\psi \in \Sy$ or $\psi \in \A$  there is a natural transport plan $P_\psi \in \rho_\psi$ defined by 
$$P_\psi=\sum_{\alpha_1, \dots, \alpha_N} |\psi|^2 (x_1,\alpha_1; \dots; x_N, \alpha_N).$$
Then for every $\rho \in H$ define $\F^\Sy_\hbar , \ \F^\A _\hbar: \Pi(\rho) \to \R^+ \cup \{+ \infty\}$ as follows:
$$ \F^\Sy_\hbar (P)= \left\{ \begin{array}{ll}
T_\hbar(\psi) + V_{ee} (\psi) & \mbox{if} \ P=P_\psi\ \mbox{for some} \ \psi \in \Sy \\
+\infty & \mbox{otherwise},
\end{array}
\right.$$
and 
$$ \F^\A_\hbar (P)= \left\{ \begin{array}{ll}
T_\hbar(\psi) + V_{ee} (\psi) & \mbox{if} \ P= P_\psi\ \mbox{for some} \ \psi \in \A \\
+\infty & \mbox{otherwise}.
\end{array}
\right.$$
It follows that 
\begin{equation}\label{refbos}
 F^\Sy _\hbar (\rho)= \min_{\Pi(\rho)} \F^\Sy_\hbar (P), 
\end{equation} 
and 
\begin{equation}\label{refferm}
 F^\A _\hbar (\rho)= \min_{\Pi(\rho)} \F^\A_\hbar (P). 
\end{equation} 
Concerning the optimal transport problem we only need to incorporate the symmetry constraint in the transport functional.
For every $\sigma\in \mathfrak{S}^N$ permutation of $\{1, \dots, N\}$ and every $P \in \Pb (\R^{Nd})$ we consider $\sigma_\sharp P$ the image measure via $\sigma$ of the measure $P$,
where with a little abuse of notations we have denoted by 
$$\sigma(x_1, \dots, x_N):= (x_{\sigma(1)}, \dots, x_{\sigma(N)}).$$
For every $P \in \Pi(\rho)$ we may consider the measure 
$$\tilde P:= \frac{1}{N!} \sum_{\sigma \in \mathfrak{S}^N} \sigma_\sharp P.$$
 We have that $\tilde P \in \Pi(\rho)$ and, since the cost is also permutation invariant the transport cost of $\tilde P$ is the same as the cost of $P$.
We say that $P$ is symmetric if $\sigma_\sharp P=P$ for every $\sigma \in \mathfrak{S}^N$. Then, for example, the measure $\tilde P$ above is symmetric. 
Define, then 
\begin{defi}
$$ \Co_\Sy(P)= \left\{ \begin{array}{ll}
\int c dP & \mbox{if} \ P \ \mbox{is symmetric}, \\
+\infty & \mbox{otherwise}.
\end{array}
\right.$$
\end{defi}
By the previous discussion 
\begin{equation}\label{reftrasp}
C(\rho)= \min_{\Pi(\rho)} \Co_\Sy (P).
\end{equation}
Then the common domain of minimization of $\F_\hbar^\Sy$, $\F^\A_\hbar$ and $\Co_\Sy$ is $\Pi(\rho)$ which we consider embedded in the space of probability measures $\Pb$ 
equipped with the tight convergence. 
\begin{defi} A generalized sequence of wave functions $\{\psi_\hbar\}$ converges to a transport plan $P$ if 
$ P_{\psi_\hbar}\rightharpoonup P$.
\end{defi}
We will prove
\begin{thm}\label{bos-gam} For every $\rho \in H$ the functionals $\F^\Sy _\hbar$ are equicoercive and 
$$ \F^\Sy _\hbar \stackrel{\Gamma}{\to }\Co_\Sy,$$  
with respect to the tight convergence of measures.
\end{thm}

\begin{thm}\label{fer-gam}
For every $\rho \in H$ the functionals $\F^\A _\hbar$ are equicoercive and for $d=2,3,4$ and $N=2,3$, 
$$ \F^\A _\hbar \stackrel{\Gamma}{\to }\Co_\Sy,$$  
with respect to the tight convergence of measures.
\end{thm}
The proofs of Theorems  \ref{bos-gam} and \ref{fer-gam} above coincide for a large part and will constitute the resto of this section.

\subsection{Equicoerciveness}

\begin{lmm} \label{prok} $\Pi(\rho)$ is compact with respect to the weak convergence. 
\end{lmm}
\begin{proof} First we prove that $\Pi(\rho)$ is tight. In fact, for $\eps > 0$ let $K \subseteq \R^d$ compact such that
\[
 \int_{K^c} \rho(x) dx \leq \eps.
\]
Observe that
\[
 \left( K^N \right)^c = \left( K^c \times \R^{d(N-1)} \right) \cup \left( \R^d \times K^c \times \R^{d(N-2)} \right) \cup \dotsb \cup \left( \R^{d(N-1)} \times K^c \right),
\]
and for every $P \in \Pi(\rho)$
\begin{align*}
 \int_{(K^N)^c} dP(X) \leq N\eps
\end{align*}
and $K^N$ is compact. By Prokhorov's Theorem we deduce that $\Pi(\rho)$ is relatively compact. However, if $P_n \wconv P$ and $P_n \in \Pi(\rho)$, for every function $\phi(x_j) \in C_b(\rnd)$ depending only on the $j$-th variable one has
\[
 \int_{\rnd} \phi(x_j) dP_n(X) = \int_{\R^d} \phi(x) \rho(x) dx
\]
which implies
\[
 \int_{\R^d} \phi(x) \rho(x) dx = \int_{\rnd} \phi(x_j) dP(X) = \int_{\R^d} \phi(x) d \pi_{\#}^j(P)(x).
\]

Since $\phi$ was arbitrary, $\pi_{\#}^j(P)(x) = \rho(x)$, and hence $P \in \Pi(\rho)$.
\end{proof}

\subsection{$\Gamma$-$\liminf$ inequality}
\begin{prop}
Let $\rho \in H$ and let  $P$ a probability measure on $\rnd$. If $\{\psi_\hbar\}_\hbar  \subseteq \Sy(\rho)$  (or $\A(\rho)$) is a generalized sequence such that 
$P_{\psi_\hbar} \wconv P$. Then
\begin{itemize} 
\item[(i)] $P$ is symmetric, 
\item[(ii)]  $\liminf_{\hbar \to 0}  \F^\Sy _\hbar (P_{\psi_\hbar}) \geq \Co(P).$
\end{itemize}
\end{prop}
\begin{proof} It is easy to see that invariance with respect to permutations is a closed condition in the space of probability measures. For the inequality:
$\liminf_{\hbar \to 0}  \F^\Sy _\hbar (P_{\psi_\hbar}) =\liminf_{\hbar \to 0} T_\hbar (\psi_\hbar) + V_{ee}(\psi_\hbar) \geq \liminf_{\hbar \to 0} V_{ee}(\psi_\hbar)\geq \Co(P).$
\end{proof}

\subsection{An approximation procedure for $P\in \Pi(\rho)$}
\begin{prop} \label{new-weak} Let $\rho \in H$, and $P \in \Pi(\rho)$ be a transport plan such that
\begin{equation}\label{P-diag}
 P|_{D_\alpha} = 0
\end{equation}
for some $\alpha > 0$.

Then there exists a family of plans $\{P_\eps\} _{\eps > 0}$ such that:
\begin{enumerate}
 \item for every $\eps > 0$, $P_\eps \in \Pi(\rho)$ and is absolutely continuous with respect to the Lebesgue measure, with density given by $\varphi_\eps^2(X)$, where $\varphi_\eps$ is a suitable $H^1$ function;
 \item $P_\eps \wconv P$ as $\eps \to 0 ;$
 \item $\limsup_{\eps \to 0} \Co(P_\eps) \leq \Co(P);$
 \item the kinetic energy of $\varphi_\eps$ is explicitly controlled:
 \[
  \int \abs{\nabla \varphi_\eps(X)}^2 dX \leq N \left( \norm{\sqrt{\rho}}^2_{H^1} + \frac{K}{4\eps^2} \right)
 \]
 for a suitable constant $K>0$. 
\end{enumerate}
\end{prop}

The proof is made of several steps.

We start regularizing by convolution. The plans we obtain are regular but do not have the same marginals as $P$.
We use the standard mollifiers  $\eta \colon \R^d \to \R$ as
\[
 \eta (u) = \left \{ \begin{array}{ll} ke^{-\frac{1}{1-\abs{u}^2}} & \abs{u} < 1 \\ 0 & \abs{u} \geq 1 \end{array} \right.
\]
where $k = k(d) > 0$ is a suitable constant, which depends only on the dimension $d$, such that $\int_{\R^d} \eta(u) du = 1$. Set also
\[
 \eta_\eps(u) = \frac{1}{\eps^d} \eta\left ( \frac{u}{\eps} \right ).
\]

The next Lemma is well known and it is a useful tool to estimate some $L^2$ norms which will appear later.

\begin{lmm} \label{mollifier-estimate}
 There exists a constant $K = K(d) > 0$, depending only on the dimension $d$, such that
 \[
  \int_{B(0,\eps)} \frac{\abs{\nabla \eta_\eps(u)}^2}{\eta_\eps(u)} du = \frac{K}{\eps^2}.
 \]
\end{lmm}

\begin{proof}
 Point out that
 \[
  \nabla \eta (u) = \left \{ \begin{array}{ll} \frac{-2ku}{(1-\abs{u}^2)^2} e^{-\frac{1}{1-\abs{u}^2}} & \abs{u} < 1 \\ 0 & \abs{u} \geq 1 \end{array} \right.
 \]
 and
 \[
   \nabla \eta_\eps (u) = \frac{1}{\eps^{d+1}} \nabla \eta \left ( \frac{u}{\eps} \right ).
 \]
 
 Now
 \begin{align*}
  \int_{B(0,\eps)} \frac{\abs{\nabla \eta_\eps(u)}^2}{\eta_\eps(u)} du &= \int_{B(0,\eps)} \frac{1}{\eps^{d+2}} \frac{\abs{\nabla \eta(u/\eps)}^2}{\eta(u/\eps)} du = \frac{1}{\eps^2} \int_{B(0,1)} \frac{\abs{\nabla \eta(v)}^2}{\eta(v)} dv \\
  {} &= \frac{1}{\eps^2} \int_{B(0,1)} \frac{4k\abs{v}^2}{(1-\abs{v}^2)^4} e^{-\frac{1}{1-\abs{v}^2}} dv = \frac{K(d)}{\eps^2}.
 \end{align*}

\end{proof}

Now we define the function
\[
 H_\eps(Y-X) = \prod_{i=1}^{N} \eta_\eps(y_i-x_i).
\]
and use it as mollifier to regularize the transport plan $P$ defining
\[
 \tilde{P}_\eps(Y) = \int H_\eps(Y-X) dP(X).
\]

Note that the marginals of $\tilde{P}_\eps$\footnote{As usual we mean the marginals of $\tilde{P}_\eps(Y)dY$.} are different from $\rho$, but may be written explicitly:

\begin{align*}
 \rho_\eps(y) &= \int \limits_{\R^{(N-1)d}} \tilde{P}_\eps(y,y_2, \dotsc, y_N) dy_2 \dots dy_N 
 \\ &= \int \limits_{\R^{(N-1)d}} \int \eta_\eps(y-x_1) \eta_\eps(y_2-x_2) \dotsm \eta_\eps(y_N-x_N) dP(X) dy_2 \dots dy_N
 \\ &= \int \eta_\eps(y-x_1) dP(X) = \int \limits_{\R^d} \rho(x) \eta_\eps(y-x) dx = (\rho * \eta_\eps) (y)
\end{align*}

\begin{lmm} \label{P-eps-diagonal}
 Let $\alpha$ be as in the statement of Proposition \ref{new-weak} and let $Y \in \rnd$ be such that  $\abs{y_i - y_j} < \alpha/2$ for some $i \neq j$, 
 and $\eps < \alpha/4$, then $\tilde{P}_\eps(Y) = 0$.
\end{lmm}

\begin{proof}
 Note that
 \[
  \tilde{P}_\eps(Y) = \int H_\eps(Y-X) \prod_{j=1}^{N} \chi_{B(y_j,\eps)}(x_j) dP(X).
 \]
 
 If $Y$ and $\eps$ are as in the statement, and $x_i \in B(y_i,\eps)$, $x_j \in B(y_j, \eps)$, then
 \[
  \abs{x_i-x_j} \leq \abs{x_i-y_i} + \abs{y_i-y_j} + \abs{y_j-x_j} \leq \alpha.
 \]
 
 The thesis follows from \eqref{P-diag}.
 
\end{proof}

Define now 
\[
 \tilde{\varphi}_\eps(Y) = \sqrt{\tilde{P}_\eps(Y)}.
\]

\begin{lmm}
 For every $\eps > 0$, $\tilde{\varphi}_\eps \in H^1(\rnd)$.
\end{lmm}

\begin{proof}
 Fix $\eps > 0$. Clearly $\tilde{\varphi}_\eps$ is $L^2$, since
 \begin{align*}
  \int \tilde{\varphi}_\eps^2(Y) dY &= \int \tilde{P}_\eps(Y) dY = \iint H_\eps(Y-X) dP(X) dY \\
  {} &= \iint H_\eps(Y-X) dY dP(X) = \int dP(X) = 1.
 \end{align*}
 
 Now we estimate $|\nabla \tilde{\varphi}_\eps|^2$. Using the Cauchy-Schwartz inequality,
 \begin{align*}
  \abs{\nabla \tilde{P}_\eps(Y)} &\leq \int \abs{\nabla H_\eps(Y-X)} dP(X) \\
  {} &\leq \sqrt{\int\frac{\abs{\nabla H_\eps(Y-X)}^2}{H_\eps(Y-X)} dP(X)} \sqrt{ \int H_\eps(Y-X) dP(X)} \\
  {} &= \sqrt{\int\frac{\abs{\nabla H_\eps(Y-X)}^2}{H_\eps(Y-X)} dP(X)} \sqrt{\tilde{P}_\eps(Y)}
 \end{align*}
 where the first integral is extended to the set where the integrand is defined, namely $\abs{x_j-y_j} < \eps \  \forall j$. Therefore, with the same convention,
 \begin{align*}
  \int \abs{\nabla \tilde{\varphi}_\eps(Y)}^2 dY &= \frac{1}{4} \int \frac{ |\nabla \tilde{P}_\eps(Y)|^2}{\tilde{P}_\eps(Y)} dY \leq \frac{1}{4} \iint \frac{\abs{\nabla H_\eps(Y-X)}^2}{H_\eps(Y-X)} dP(X) dY \\
  {} &= \frac{1}{4} \sum_{j=1}^N \iint \frac{\abs{\nabla_j H_\eps(Y-X)}^2}{H_\eps(Y-X)} dY dP(X) \\
  {} &\leq \frac{1}{4} \sum_{j=1}^N \iint \frac{\abs{\nabla \eta_\eps(y_j-x_j)}^2}{\eta_\eps(y_j-x_j)} dy_j dP(X) \\
  {} &= \frac{N}{4} \int \frac{\abs{\nabla \eta_\eps(u)}^2}{\eta_\eps(u)} du = \frac{KN}{4\eps^2}.
 \end{align*}
\end{proof}

In the next step we introduce a natural tecnique to get back the original marginals $\rho$ without losing too much regularity. 
This technique is original and different from the one presented in \cite{cotar2013density}. We point out that, in a 
different context, this construction may well be generalized to a plan with different marginals $\rho^1, \dotsc, \rho^N$.

The construction fits in the general scheme of composition of transport plans as presented in \cite{ambrosio2005}.

For $x, y \in \R^d$ define
\[
 \gamma_\eps(x,y) := \frac{\rho(x)\eta_\eps(y-x)}{\rho_\eps(y)}
\]
with the convention that it is zero if $\rho_\eps(y) = 0$. The two variables function $\rho(x)\eta_\eps(y-x)$ is the key point of the construction, since it has the following property that links the different marginals:
\[
 \int_{\R^d} \rho(x)\eta_\eps(y-x) dx = \rho_\eps(y), \  \int_{\R^d} \rho(x)\eta_\eps(y-x) dy = \rho(x).
\]

To simplify the notation we set
\[
 m_\eps(X, Y) = \prod_{i=1}^{N} \gamma_\eps (x_i, y_i),
\]
\[
 m_\eps^{j}(X, Y) = \prod_{\substack{i = 1 \\ i \neq j}}^{N} \gamma_\eps (x_i, y_i)
\]

Note that
\[
 \int\limits_{\R^d} \gamma_\eps(x,y) dx = \chi_{\rho_\eps > 0}(y).
\]

Now set
\[
 \Gamma_\eps(X,Y) := \tilde{\varphi}_\eps^2(Y) m_\eps(X, Y).
\]
and observe that
\[
 \int \Gamma_\eps(X,Y) dX = \tilde{\varphi}_\eps^2(Y) \prod_{i = 1}^{N} \chi_{\rho_\eps > 0}(y_i) = \tilde{\varphi}_\eps^2(Y) = \tilde{P}_\eps(Y)
\]
where we used the following remark, which will be implicit from now on:

\begin{rmk}
 If $Y$ is such that $\rho_\eps(y_i) = 0$, then
 \[
  0 = \rho_\eps(y_i) = \int_{\R^{(N-1)d}} \tilde{\varphi}_\eps(y_1, y_2, \dotsc, y_n)^2 dy_1 \dots \hat{dy_i} \dots dy_N
 \]
 and hence $\tilde{\varphi}_\eps(Y) = 0$.
\end{rmk}

Define
\[
 P_\eps(X) = \int \Gamma_\eps(X,Y) dY
\]
and calculate the marginals of $P_\eps$ to get
\begin{align*}
 \int\limits_{\R^{(N-1)d}} P_\eps(X) dx_2\dots dx_N &= \int\limits_{\R^{(N-1)d}} \int \Gamma_\eps(X,Y) dY dx_2\dots dx_N
 \\ &= \int \tilde{\varphi}_\eps^2(Y)\gamma_\eps(x_1,y_1) dY
 \\ &= \int\limits_{\R^d} \rho_\eps(y_1) \gamma_\eps(x_1, y_1) dy_1
 \\ &= \rho(x_1) 
\end{align*}
and similarly also the other $N-1$ marginals are equal to $\rho$.

\begin{lmm} \label{tilde-P-eps-diagonal}
 Let $\alpha$ be as in the statement of Proposition \ref{P-diag}. If $\abs{x_i-x_j} < \alpha/4$ for some $i \neq j$, and $\eps < \alpha/8$, then $P_\eps(X) = 0$.
\end{lmm}

\begin{proof}
 Fix $X$ and $\eps$ as in the statement, and suppose $m_\eps(X,Y) > 0$. Then necessarily $\abs{y_i-x_i} < \eps$ and $\abs{y_j-x_j} < \eps$, so that
 \[
  \abs{y_i-y_j} \leq \abs{y_i-x_i} + \abs{x_i-x_j} + \abs{x_j-y_j} \leq \alpha/2.
 \]
 
 From Lemma \ref{P-eps-diagonal} it follows that $\tilde{\varphi}_\eps^2(Y) = 0$.

\end{proof}

We define now the function $\varphi_{\eps}(X) = \sqrt{P_\eps(X)}$, and proceed to estimate its kinetic energy.

\paragraph{Estimate for the kinetic energy} Calculate first the gradient with respect to $x_j$ of $P_\eps$:

\begin{IEEEeqnarray*}{rCl}
 \nabla_j P_\eps(X) &= &\nabla_j \int \tilde{\varphi}_\eps^2(Y) m_\eps^j(X, Y) \frac{\rho(x_j)\eta_\eps(y_j - x_j)}{\rho_\eps(y_j)} dY
 \\ & {} = &\int \tilde{\varphi}_\eps^2(Y) m_\eps^j(X, Y) \frac{\nabla \rho(x_j)\eta_\eps(y_j - x_j)}{\rho_\eps(y_j)} dY 
 \\ && {} - \int \tilde{\varphi}_\eps^2(Y) m_\eps^j(X, Y) \frac{\rho(x_j) \nabla \eta_\eps(y_j - x_j)}{\rho_\eps(y_j)} dY
 \\ & {} = &A(X) + B(X).
\end{IEEEeqnarray*}

We define for simplicity
\[
 J(X) = \int \tilde{\varphi}_\eps^2(Y) m_\eps^j(X, Y) \frac{\eta_\eps(y_j - x_j)}{\rho_\eps(y_j)} dY
\]
so that, for example,
\[
 P_\eps(X) = \rho(x_j) J(X),
\]
\[
 A(X) = \nabla \rho(x_j) J(X).
\]

Now
\[
 \nabla_j \varphi_\eps(X) = \nabla_j \sqrt{P_\eps}(X) = \frac{\nabla_j P_\eps(X)}{2 \sqrt{P_\eps(X)}} = \frac{1}{2 \sqrt{P_\eps(X)}} [A(X) + B(X)]
\]
and we estimate the square of the  $L^2$ norm of every term.

\begin{enumerate}
 \item 
  \[
   \int \frac{\abs{A(X)}^2}{P_\eps(X)} dX = \int\limits_{\R^d} \frac{\abs{\nabla \rho(x_j)}^2}{\rho(x_j)} dx_j \leq 4 \norm{\sqrt{\rho}}_{H^1}^2.
  \]
 
 \item By Cauchy-Schwartz inequality,
 \[
  \abs{B(X)}^2 \leq \rho(x_j)J(X) \cdot \int \tilde{\varphi}_\eps(Y)^2 m_\eps^j(X, Y) \frac{\rho(x_j) \abs{\nabla \eta_\eps(y_j - x_j)}^2}{\rho_\eps(y_j) \eta_\eps(y_j-x_j)} dY.
 \]
 (Here the integral is extended to the region where $\eta_\eps(y_j-x_j) > 0$.) Therefore, with the same convention,
 \begin{align*}
  \int \frac{\abs{B(X)}^2}{P_\eps(X)} dX &\leq \iint \tilde{\varphi}_\eps(Y)^2 m_\eps^j(X, Y) \frac{\rho(x_j) \abs{\nabla \eta_\eps(y_j - x_j)}^2}{\rho_\eps(y_j) \eta_\eps(y_j-x_j)} dY dX
  \\ &= \int\limits_{\R^d} \int \tilde{\varphi}_\eps(Y)^2 \frac{\rho(x_j) \abs{\nabla \eta_\eps(y_j - x_j)}^2}{\rho_\eps(y_j) \eta_\eps(y_j-x_j)} dY dx_j
  \\ &= \int\limits_{\R^d} \int\limits_{\R^d} \frac{\rho(x_j) \abs{\nabla \eta_\eps(y_j - x_j)}^2}{\eta_\eps(y_j-x_j)} dy_j dx_j = \int\limits_{\R^d} \frac{\abs{\nabla \eta_\eps(y)}^2}{\eta_\eps(y)} dy = \frac{K}{\eps^2}.
 \end{align*}
\end{enumerate}

Moreover,
\begin{align*}
 \int \frac{A(X)\cdot B(X)}{4 P_\eps(X)} dX &= -\frac{1}{4} \iint \frac{\nabla \rho(x_j) \cdot \nabla \eta_\eps(x_j-y_j)}{\rho_\eps(y_j)} \tilde{\varphi}_\eps^2(Y) m_\eps^j(X,Y) dY dX \\
 & {} = -\frac{1}{4} \iint_{\R^d \times \R^d} \nabla \rho(x_j) \cdot \nabla \eta_\eps(x_j-y_j) dx_j dy_j \\
 & {} = -\frac{1}{4} \left( \int_{\R^d} \nabla \rho(x) dx \right) \cdot \left( \int_{\R^d} \nabla \eta_\eps(z) dz \right)
\end{align*}
and the second factor is equal to 0, as is easy to see integrating in spherical coordinates.

Thus
\[
 \int \abs{\nabla_j \varphi_\eps(X)}^2 dX \leq \frac{1}{4} \int \frac{\abs{A(X)}^2 + \abs{B(X)}^2}{P_\eps(X)} dX \leq \norm{\sqrt{\rho}}^2_{H^1} + \frac{K}{4\eps^2},
\]
and by summation over $j$
\[
 \int \abs{\nabla \varphi_\eps(X)}^2 dX \leq N \left( \norm{\sqrt{\rho}}^2_{H^1} + \frac{K}{4\eps^2} \right).
\]

Here we prove (ii). First we give the following lemma, which specifies that, for $\eps \to 0$, the mass of $\tilde{P}_\eps$ is concentrated near the mass of $P$.

\begin{lmm} \label{distant-mass} Suppose $R, \delta > 0$ are such that
 \[
  \int \limits_{\abs{X} > R} dP(X) \leq \delta;
 \]
 then, if $\eps\sqrt{N} < R$,
 \[
  \int \limits_{\abs{Y} > 2R} \tilde{P}_\eps(Y) dY \leq \delta.
 \]
\end{lmm}

\begin{proof}
 \begin{align*}
  \int \limits_{Y > 2R} \tilde{P}_\eps(Y) dY &= \iint \limits_{\abs{Y} > 2R} H_\eps(Y-X) dP(X) dY
  \\ &= \iint \limits_{\substack{\abs{Y} > 2R} \cap \abs{X} > R} H_\eps(Y-X) dP(X) dY
  \\ &\leq  \iint \limits_{\abs{X} > R} H_\eps(Y-X) dY dP(X) = \int \limits_{\abs{X} > R} dP(X) \leq \delta 
 \end{align*}
 since, where $\abs{X} \leq R$, one has $\abs{Y-X} \geq R > \eps\sqrt{N}$, and hence there exists $i$ such that $\abs{x_i-y_i} > \eps$.
\end{proof}

Next, to prove that $P_\eps \wconv P$, we interpolate $\tilde{P}_\eps$ in between:

\begin{lmm} \label{weak-convergence-P} $\tilde{P}_\eps \wconv P$.
 
\end{lmm}

\begin{proof} Let $\phi \in C_b(\rnd)$;
\begin{align*}
 \abs{\int \phi(Y) d\tilde{P}_\eps(Y) - \int \phi(X) dP(X) } &= \abs{ \iint [\phi(Y) - \phi(X)] H_\eps(Y-X) dP(X)dY }
 \\ &\leq \iint \abs{\phi(Y) - \phi(X)} H_\eps(Y-X) dP(X)dY.
\end{align*}

Given $\delta > 0$, let $R$ be such that the hypotesis of Lemma \ref{distant-mass} holds. We divide $\rnd \times \rnd$ in three disjoint regions:
\[
 E_1 = \gra{\abs{X} > R} \ E_2 = \gra{\abs{X} \leq R, \abs{Y} \leq 2R} \ E_3 = \gra{\abs{X} \leq R, \abs{Y} > 2R}.
\]

As before, if $\eps\sqrt{N} < R$, on $E_3$ one has $H_\eps(X-Y) \equiv 0$.

\begin{align*}
 \iint \limits_{E_1} \abs{\phi(Y) - \phi(X)} H_\eps(Y-X) dP(X)dY &\leq 2\norm{\phi}_{\infty} \iint \limits_{E_1} H_\eps(Y-X) dP(X)dY
 \\ &\leq 2 \delta \norm{\phi}_{\infty}.
\end{align*}

On the other hand, $E_2$ is compact; take $\eps_0$ such that $\abs{X-Y} \leq \eps_0$ implies $\abs{\phi(X)-\phi(Y)} \leq \delta$. If $\eps\sqrt{N} \leq \eps_0$ we get
\begin{align*}
 \iint \limits_{E_2} \abs{\phi(Y) - \phi(X)} H_\eps(Y-X) dP(X)dY &\leq \delta \iint \limits_{E_2} H_\eps(Y-X) dP(X)dY
 \\ &\leq \delta \iint H_\eps(Y-X) dP(X)dY = \delta.
\end{align*}

\end{proof}

\begin{lmm} \label{weak-convergence-full} $P_\eps \wconv P$.
 \end{lmm}

\begin{proof} Let $\phi \in C_b(\rnd)$. Using the fact that $\tilde{P}_\eps \wconv P$ (Lemma \ref{weak-convergence-P}), it is left to estimate
\begin{align*}
 \abs{\int \phi(X) P_\eps(X) dX - \int \phi(Y) \tilde{P}_\eps(Y) dY} &= \abs{\iint [\phi(X) - \phi(Y)] \Gamma_\eps(X,Y) dX dY}
 \\ & \leq \iint \abs{\phi(X) - \phi(Y)} \Gamma_\eps(X,Y) dX dY
\end{align*}

As in the proof of Lemma \ref{weak-convergence-P}, given $\delta>0$ let $R$ be such that the hypotesis of Lemma \ref{distant-mass} holds. We divide $\rnd \times \rnd$ in three disjoint regions:
\[
 E_1 = \gra{\abs{Y} > 2R} \  E_2 = \gra{\abs{Y} \leq 2R, \abs{X} \leq 3R} \  E_3 = \gra{\abs{Y} \leq 2R, \abs{X} > 3R}.
\]

If $\eps\sqrt{N} < R$, as before, on $E_3$ the integral is zero since $\Gamma_\eps(X, Y) \equiv 0$ there. Thanks to Lemma \ref{distant-mass},
\begin{align*}
 \iint \limits_{E_1} \abs{\phi(X) - \phi(Y)} \Gamma_\eps(X,Y) dX dY &\leq 2 \norm{\phi}_{\infty} \iint \limits_{E_1} \Gamma_\eps(X,Y) dX dY
 \\ &= 2\norm{\phi}_{\infty} \int \limits_{\gra{\abs{Y} > 2R}} \tilde{P}_\eps(Y) dY \leq 2 \delta \norm{\phi}_{\infty}
\end{align*}

Exactly as before, using that $E_2$ is compact and $\phi$ is absolutely continuous the thesis follows.

\end{proof}

It is left to prove (iii). The cost function $c$ is not continuous neither bounded. However, recall Lemma \ref{tilde-P-eps-diagonal}, it is bounded on the complement of $D_{\alpha/4}$. With this in mind, consider the function $v \colon \R^{2d} \to \R$ defined as
\[
 v(x,y) = \left \{ \begin{array}{ll} \displaystyle \frac{1}{\abs{x-y}} & \text{if } \abs{x-y} \geq \alpha/4 \\ \displaystyle 4/\alpha & \text{elsewhere}. \end{array} \right.
\]
and set
\[
 \overline{c}(X) = \sum_{1 \leq i < j \leq N} v(x_i,x_j).
\]

Clearly $\overline{c}(X) \leq c(X)$, and  $\overline{c}$ is continuous (sum of continuous functions), bounded by $\tbinom{N}{2}\tfrac{4}{\alpha}$; moreover, thanks to the property \eqref{P-diag} and Lemma \ref{tilde-P-eps-diagonal},
\[
 \int \overline{c}(X) dP(X) = \int c(X) dP(X), \  \int \overline{c}(X) P_\eps (X) dX = \int c(X) P_\eps (X) dX.
\]

We can conclude the estimate as follows:
\[
 \limsup_{\eps \to 0} C(P_\eps) = \limsup_{\eps \to 0} \int \overline{c}(X) P_\eps(X) dX = \int \overline{c}(X) dP(X) = \int c(X) dP(X).
\]
\qed

Proposition  \ref{new-weak} may be extended to all plans.
\begin{prop}\label{new-weak-full}  Let $P \in \Pi(\rho)$  (not necessarily satisfying the property \eqref{P-diag}). 
Then there exists a family of plans $\{P_\eps\} _{\eps > 0}$ such that:
\begin{enumerate}
 \item for every $\eps > 0$, $P_\eps \in \Pi(\rho)$ and is absolutely continuous with respect to the Lebesgue measure, with density given by $\varphi_\eps^2(X)$, where $\varphi_\eps$ is a suitable $H^1$ function;
 \item $P_\eps \wconv P$ as $\eps \to 0;$
 \item $\limsup_{\eps \to 0} \Co(P_\eps) \leq \Co(P);$
 \item the kinetic energy of $\varphi_\eps$ is explicitly controlled:
 \[
  \int \abs{\nabla \varphi_\eps(X)}^2 dX \leq N \left( \norm{\sqrt{\rho}}^2_{H^1} + \frac{K}{4\eps^2} \right)
 \]
 for a suitable constant $K>0$. 
\end{enumerate}
\end{prop}
\begin{proof}
 If $C(P) = \infty$ there is nothing to prove. Suppose $C(P) = K < \infty$. Let $r >0$ be a parameter, and split
\[
 P = Q_r + P|_{D_r}.
\]
Let $\sigma_r$ be the marginals of $Q_r$, and $\tilde{\rho}_r$ those of $P|_{D_r}$; clearly $\sigma_r + \tilde{\rho}_r = \rho$. Since $\rho \in L^1 \cap L^{\frac{d}{d-2}}$ by Sobolev embedding, and $\tilde{\rho}_r \leq \rho$ pointwise, we have $\tilde{\rho}_r \in L^1 \cap L^{\frac{d}{d-2}}$
Although $\tilde{\rho}_r$ needs not to be a probability measure on $\R^d$, we can suppose there exists $\lambda_r > 0$ such that
\[
 \int_{\R^d} \tilde{\rho}_r (x) dx = \frac{1}{\lambda_r} < 1,
\]
otherwise $P|_{D_r} = 0$ and we get the result directly by Proposition \ref{new-weak}.
Let now $\tilde{P}_r$ be an optimal transport plan in $\Pi(\lambda_r \tilde{\rho}_r)$, and define
\[
 P_r = Q_r + \frac{\tilde{P}_r}{\lambda_r},
\]
which lies in $\Pi(\rho)$. 
On the one hand we have the following
\begin{lmm} \label{P_r-P}
 $P_r \wconv P$. 
\end{lmm}
\begin{proof} Recall that $C(P)$ is finite to get
\[
 K = C(P) \geq C(P|_{D_r}) \geq \frac{1}{r} P(D_r),
\]
and \emph{a fortiori} for $\tilde{P}_r/\lambda_r$ due to the optimality. Hence
\[
 \lim_{r \to 0} P(D_r) = \lim_{r \to 0} \frac{\tilde{P}_r(D_r)}{\lambda_r} = 0.
\]

Take $f \in C_b(\rnd)$, and estimate
 \begin{align*}
  \abs{\int f(X) dP_r(X) - \int f(X) dP(X)} &= \abs{\int f(X) dP|_{D_r}(X) - \frac{1}{\lambda_r} \int f(X) d\tilde{P}_r(X)} \\
  {} &\leq \norm{f}_{\infty} \left(P(D_r) + \frac{\tilde{P}_r(D_r)}{\lambda_r} \right) \to 0
 \end{align*}
 as $r \to 0$.

\end{proof}

On the other hand,
\[
 C(P_r) = C(Q_r) + \frac{C(\tilde{P}_r)}{\lambda_r} \leq C(Q_r) + C(P|_{D_r}) = C(P), 
\]
thus
\[
 \limsup_{r \to 0} C(P_r) \leq C(P).
\]

Thanks to Proposition \ref{L1-L3} $C(\tilde{P}_r)$ is finite, and by Theorem \ref{de-pascale} there exists $\alpha = \alpha(r) > 0$ such that $\tilde{P}_r$ is supported outside $D_\alpha$.\footnote{Observe that $\alpha(r)$ may be chose decreasing as $r \to 0$, as follows from Theorem \ref{de-pascale}.} Recall now Proposition \ref{new-weak} to find $\varphi_{\eps,r}$ weakly converging to $P_r$ as $\eps \to 0$, with
\[
 \int \abs{\nabla \varphi_{\eps,r} (X)}^2 dX \leq N \left( \norm{\sqrt{\rho}}^2_{H^1} + \frac{K}{4\eps^2} \right)
\]
and
\[
 \limsup_{\eps \to 0} C(\varphi_{\eps,r}) = C(P_r).
\]

It suffices now to take $\gra{\varphi_{r,r}}_{r > 0}$ to get the thesis of Proposition \ref{new-weak} for general plans. In fact, (i), (iii) and (iv) are already clear; given $\delta>0$, let $R$ be such that
\[
 \int_{\gra{\abs{X} > R}} dP_r (X) \leq \delta. 
\]

Note that $R$ may be chose indipendent from $r$, since the marginals of $P_r$ are all equal to $\rho$, and we may choose $K \subseteq \R^d$ compact such that
\[
 \int_{K} \rho(x) dx \leq \frac{\delta}{N},
\]
and $R$ sufficiently large such that $K^N \subseteq B(0,R)_{\rnd}$.

Now to prove weak convergence take $\phi \in C_b(\rnd)$ and proceed as in the previous paragraph to estimate
\[
 \abs{\int \phi(X) \varphi^2_{r,r}(X) dX - \int \phi(X) dP_r(X)},
\]
using in addiction Lemma \ref{P_r-P} to estimate
\[
 \abs{\int \phi(X) dP_r(X) - \int \phi(X) dP(X)}.
\]

\end{proof}

\begin{rmk} Carefully following the constructions in Propositions 3.7 one may observe that if $P$ is permutations invariant then the approximating $P_\varepsilon$ have the same property. Some care is needed in Proposition 3.16 where, when choosing the minimizer $\tilde{P}_r$ on the strip around the diagonal, one has to choose the symmetric minimizer.
\end{rmk}

\subsection{Constructing wave-functions}

A wave-function depends on $N$ space-spin variables. In the previous subsections we worked mainly in $\rnd$, since we were considering transport plans in $\Pi(\rho)$. To introduce 
the spin we will separate the spin dependence as follows: for every $s$ binary string of length $N$ we consider the function $\psi_s(x_1, \dotsc, x_N) = \psi(x_1, s_1, \dotsc, x_N, s_N)$, then we describe $\psi$ as a $2^N$-dimensional vector
\[
 \boldsymbol{\psi}(z_1, \dotsc, z_N) = \left( \psi_s(X) \right)_{s \in S}.
\]

As an example, if $N = 2$ we would have
\[
 \boldsymbol{\psi}(z_1, z_2) = \begin{pmatrix} \psi_{00}(x_1, x_2) \\ \psi_{01}(x_1, x_2) \\ \psi_{10}(x_1, x_2) \\ \psi_{11}(x_1, x_2) \end{pmatrix},
\]
and for $N=3$ a wave-function would be represented as
\[
 \boldsymbol{\psi}(z_1, z_2, z_3) = \begin{pmatrix} \psi_{000}(x_1, x_2,x_3) \\ \psi_{001}(x_1, x_2,x_3) \\ \psi_{010}(x_1, x_2,x_3) \\ \psi_{100}(x_1, x_2,x_3) \\ \psi_{110}(x_1,x_2,x_3) \\ \psi_{101}(x_1,x_2,x_3) \\ \psi_{011}(x_1,x_2,x_3) \\ \psi_{111}(x_1,x_2,x_3) \end{pmatrix}.
\]

Note that now the density $\abs{\psi(X)}^2$ is simply the square of the Euclidean norm of the vector $\boldsymbol{\psi}$, and the same holds for $\nabla \boldsymbol{\psi}$, once we set
\[
 \nabla \boldsymbol{\psi}(z_1,\dotsc,z_N) = \left( \nabla \psi_s(X) \right)_{s \in S}.
\]

Let us take now a fermionic (i.e., antisymmetric) wave-function $\psi$, and consider a spin state $s = (s_1,\dotsc,s_N)$. If $i < j$ are such that $s_i = s_j$, consider $\sigma = (i\;j) \in \mathfrak{S}_N$ to get
\begin{align*}
 \psi_s (x_1,\dotsc,x_N) &= \mbox{sgn}(\sigma) \psi_{\sigma(s)} (x_{\sigma(1)},\dotsc,x_{\sigma(N)}) \\
 {} &= -\psi_s(x_1,\dotsc,x_j,\dotsc,x_i,\dotsc,x_N).
\end{align*}
Hence we get the following
\begin{rmk}
 If $\psi$ is fermionic and $s$ is a spin state, $\psi_s$ is \emph{separately} antisymmetric with respect to the spatial variables such that $s_j = 0$, and with respect to the spatial variables such that $s_j = 1$.
\end{rmk}

Consider now two spin states $s$ and $s'$ with the same number of ones and zeroes. Then $\psi_s$ and $\psi_{s'}$ are related:
 taking $\sigma \in \mathfrak{S}_N$ such that $\sigma(s) = s'$, we get
\[
 \psi_{s'}(x_{\sigma(1)},\dotsc,x_{\sigma(N)}) = \mbox{sgn}(\sigma) \psi_s (x_1, \dotsc, x_N).
\]
These observations will be implicit in the following.

\subsection{Fermionic wave-functions with given density}

Suppose we have $\alpha > 0$ and a symmetric funtion $\psi(x_1,\dotsc,x_N)\geq 0$ of $H^1$ with the property that
\begin{equation} \label{diag-away}
 \psi(X) = 0 \smallspace \text{if $\abs{x_i-x_j} < \alpha$ for some $i\neq j$}.
\end{equation}

We wonder if there exists a fermionic wave function $\boldsymbol{\psi}$ such that
\[
 \sum_{s\in S} \abs{\boldsymbol{\psi}(x_1,s_1,\dotsc,x_N,s_N)}^2 = \psi^2(X),
\]
and
\[
 \norm{\boldsymbol{\psi}}_{H^1} \leq C \norm{\psi}_{H^1}
\]
for a suitable constant $C$. In fact, we managed to prove the following

\begin{prop} \label{fermionic} For $N = 2, 3$, $d=3,4$, given $\psi \in H^1$ symmetric, with $\psi |_{D_\alpha} = 0$ for some $\alpha > 0$, there exists $\bpsi$ fermionic such that
 \[
  \sum_{s \in S} \abs{\bpsi_s(X)}^2 = \psi^2(X)
 \]
 and
 \[
  \sum_{s \in S} \abs{\nabla \bpsi_s(X)}^2 \leq \abs{\nabla \psi(X)}^2 + \frac{C}{\alpha^2} \psi^2(X).
 \]
\end{prop}

In \cite{cotar2013density}, the following $\bpsi$ is given as a wavefunction:
\begin{align*}
 \bpsi_{00} &= 0 \\
 \bpsi_{01} &= \psi(x_1,x_2) \\
 \bpsi_{10} &= -\psi(x_1,x_2) \\
 \bpsi_{11} &= 0,
\end{align*}
which is in fact fermionic with bounded kinetic energy. Note, however, that this construction cannot work for a larger number of particles. Indeed, if a binary string $s$ has length $N\geq 3$, then there are at least two ones, or two zeros, on places $i\neq j$ -- thus the corresponding function $\bpsi_s$ must change sign for a suitable flip of the variables (namely, $x_i \mapsto x_j$, $x_j \mapsto x_i$). We will exhibit a wave-function $\bpsi$ (with square density $\psi^2$) such that $\bpsi_{10} = \bpsi_{01} = 0$. This forces $\bpsi_{00}$ and $\bpsi_{11}$ to be different from zero and antisymmetric. We remark also that, for this kind of construction, the condition \eqref{diag-away} is ``morally necessary''.

\subsection{Construction for $N=2$, $d=3$}

The variable $X$ will be expanded as $X = (x,y) = (x_1,x_2,x_3; y_1,y_2,y_3)$. Set $r = \alpha/\sqrt{3}$, and take $\xi = -\frac{1}{2} + i\frac{\sqrt{3}}{2}$ a primitive cube root of 1. The key point is to choose two auxiliary $C^{\infty}$ functions $a,b \colon \R \to \R$ such that
 \begin{itemize}
  \item[(i)] $a^2+b^2=1$;
  \item[(ii)] $b$ is symmetric, $b(t) = 0$ if $\abs{t} \geq r$;
  \item[(iii)] $a$ is antisymmetric, $a(t) = -1$ if $t \leq -r$, $a(t) = 1$ if $t \geq r$;
  \item[(iv)] $\abs{a'}, \abs{b'} \leq k/r$.
 \end{itemize}
 
Note that the constant $k>1$ may be chosen arbitrarily close to 1. To shorten the notation we set $a_j = a(x_j-y_j)$, $b_j = b(x_j-y_j)$ for $j=1,2,3$. Now define
\begin{align*}
 g_1(x,y) &= \frac{1}{\sqrt{3}} (a_1+b_1 a_2+ b_1 b_2 a_3) \\
 g_\xi(x,y) &= \frac{\sqrt{2}}{\sqrt{3}} (a_1+\xi b_1 a_2+\xi^2 b_1 b_2 a_3).
\end{align*}

By direct computation one sees that $\abs{g_1}^2 + \abs{g_\xi}^2 = a_1^2 + b_1^2 a_2^2 + b_1^2 b_2^2 a_3^2$, since clearly $\abs{a_1+b_1 a_2+ b_1 b_2 a_3 } = \abs{a_1+\xi b_1 a_2+\xi^2 b_1 b_2 a_3 }$. 
Next we define the wave-function
\begin{align*}
 \bpsi_{00}(x,y) &= g_1(x,y) \psi(x,y) \\
 \bpsi_{01}(x,y) &= 0 \\
 \bpsi_{10}(x,y) &= 0 \\
 \bpsi_{11}(x,y) &= g_\xi (x,y) \psi(x,y). \\
\end{align*}

The following equality is crucial in the construction

\begin{equation} \label{claim}
\psi^2(x,y) b_1^2 b_2^2 a_3^2 = \psi^2(x,y) b_1^2 b_2^2.
\end{equation}
 
This holds because where $\abs{a_3}^2 = 1$, i.e. where $\abs{x_3-y_3} \geq r$, the equality holds. It also holds where $b_1 = 0$ or $b_2 = 0$, 
i.e. where $\abs{x_1-y_1} \geq r$ or $\abs{x_2-y_2} \geq r$. It is left the region where $\abs{x_j-y_j} \leq r$ for every $j = 1,2,3$, but there it holds
\[
 \abs{x-y} = \sqrt{\abs{x_1-y_1}^2 + \abs{x_2-y_2}^2 + \abs{x_3-y_3}^2} \leq r\sqrt{3} = \alpha,
\]
and hence the equality holds because $\psi^2(x,y) = 0$.

Now one can compute
\begin{align*}
 \abs{\bpsi_{00}}^2 + \abs{\bpsi_{11}}^2 &= \left( \abs{g_1(x,y)}^2 + \abs{g_\xi(x,y)}^2 \right) \psi^2(x,y) \\
 {} &= \left( a_1^2 + b_1^2 a_2^2 + b_1^2 b_2^2 a_3^2 \right) \psi^2(x,y) \\
 {} &= \left( a_1^2 + b_1^2 a_2^2 + b_1^2 b_2^2 \right) \psi^2(x,y) \\
 {} &= \left( a_1^2 + b_1^2 \right) \psi^2(x,y) = \psi^2(x,y).
\end{align*}

Next come the estimates for the derivatives. Since
\begin{align*}
 \nabla \bpsi_{00}(x,y) &= \psi(x,y) \nabla g_1(x,y) + g_1(x,y) \nabla \psi(x,y) \\
 \nabla \bpsi_{11}(x,y) &= \psi(x,y) \nabla g_\xi(x,y) + g_\xi(x,y) \nabla \psi(x,y),
\end{align*}
it follows
\begin{align*}
 \abs{\nabla \bpsi(x,y)}^2 = &\abs{\nabla \psi(x,y)}^2 + \left( \abs{\nabla g_1(x,y)}^2 + \abs{\nabla g_\xi (x,y)}^2 \right) \psi^2(x,y) \\
 {} &+ \psi(x,y) \nabla \psi(x,y) \cdot v(x,y)
\end{align*}
where
\[
 v(x,y) = 2g_1(x,y) \nabla g_1(x,y) + g_{\bar{\xi}}(x,y) \nabla g_\xi(x,y) + g_\xi(x,y) \nabla g_{\bar{\xi}}(x,y).
\]
We claim that $\psi(x,y) v(x,y) = 0$. Again by direct computation one gets
\[
 v = 6 [ a_1\nabla a_1 + b_1 a_2 \nabla (b_1 a_2) + b_1 b_2 a_3 \nabla (b_1b_2a_3) ]. 
\]
Since the next steps work for general $d$, we group the result in the following

\begin{lmm} \label{no-crossed} Let $a_j,b_j$ be defined as before and evaluated in the point $x-y$. Then
\[
 [a_1 \nabla a_1 + (b_1a_2) \nabla (b_1a_2) + \dotsb + (b_1b_2 \dotsm b_{d-1}a_d) \nabla (b_1b_2\dotsm b_{d-1}a_d)] \psi(x,y) = 0.
\]
\end{lmm}

\begin{proof} Observe that
\begin{equation} \label{gradients}
 \begin{array}{rcl}
  \nabla a_1(x,y) &= &(a'(x-y),0,\dotsc,0,-a'(x-y),0,\dotsc,0) \\
  \nabla b_1(x,y) &= &(b'(x-y),0,\dotsc,0,-b'(x-y),0,\dotsc,0)
 \end{array}
\end{equation}
and similarly for the other gradients. Moreover, from $a^2+b^2 = 1$ it follows $aa'+bb' = 0$, while $\psi b_1 b_2 \dotsm b_{d-1} \nabla a_d = 0$ and $\psi b_1b_2 \dotsm b_{d-1} a_d^2 = \psi b_1b_2 \dotsm b_d$ for the same reason as in claim \ref{claim}. Hence we have
\begin{IEEEeqnarray*}{rCl}
 \IEEEeqnarraymulticol{3}{l}{[a_1 \nabla a_1 + (b_1a_2) \nabla (b_1a_2) + \dotsb + (b_1b_2 \dotsm b_{d-1}a_d) \nabla (b_1b_2\dotsm b_{d-1}a_d)] \psi} \\
 \quad &= &[a_1 \nabla a_1 + (b_1a_2) \nabla (b_1a_2) + \dotsb + (b_1b_2 \dotsm b_{d-1}) \nabla (b_1b_2\dotsm b_{d-1})] \psi.
\end{IEEEeqnarray*}

A ``chain reaction'' is now generated by the following formula, valid for every $k \geq 1$:
\begin{IEEEeqnarray*}{rCl}
 \IEEEeqnarraymulticol{3}{l}{(b_1\dotsm b_k a_{k+1}) \nabla (b_1 \dotsm b_k a_{k+1}) + (b_1\dotsm b_k b_{k+1}) \nabla (b_1 \dotsm b_k b_{k+1})} \\
 \quad &= &(b_1\dotsm b_k a^2_{k+1}) \nabla (b_1 \dotsm b_k) + (b^2_1\dotsm b^2_k) a_{k+1} \nabla a_{k+1} \\
 && {} + (b_1\dotsm b_k b^2_{k+1}) \nabla (b_1 \dotsm b_k) + (b^2_1\dotsm b^2_k) b_{k+1} \nabla b_{k+1} \\
 &= &(b_1\dotsm b_k) \nabla (b_1 \dotsm b_k).
\end{IEEEeqnarray*}
\end{proof}
It is left to estimate $\left( \abs{\nabla g_1(x,y)}^2 + \abs{\nabla g_\xi (x,y)}^2 \right) \psi(x,y)$. Again we compute directly
\begin{align*}
 \abs{\nabla g_1(x,y)}^2 + \abs{\nabla g_\xi (x,y)}^2 &= 3\left( \abs{\nabla a_1}^2 + \abs{\nabla (b_1 a_2)}^2 + \abs{\nabla (b_1 b_2 a_3)}^2 \right).
\end{align*}

Note, however, that because of \eqref{gradients} we have $\nabla a_i \cdot \nabla b_j = 0$ and $\nabla b_i \cdot \nabla b_j = 0$ if $i \neq j$. Therefore,
\begin{align*}
 \abs{\nabla g_1(x,y)}^2 + \abs{\nabla g_\xi (x,y)}^2 = &\abs{\nabla a_1}^2 + b_1^2 \abs{\nabla a_2}^2 + a_2^2 \abs{\nabla b_1}^2 + b_2^2 a_3^2 \abs{\nabla b_1}^2 \\
 {} &+ b_1^2 a_3^2 \abs{\nabla b_2}^2 + b_1^2 b_2^2 \abs{\nabla a_3}^2.
\end{align*}
and, using again the idea of claim \eqref{claim},
\begin{IEEEeqnarray*}{rCl}
  \IEEEeqnarraymulticol{3}{l}{\left( \abs{\nabla g_1(x,y)}^2 + \abs{\nabla g_\xi (x,y)}^2 \right) \psi^2(x,y)} \\
  \quad &=& \left( \abs{\nabla a_1}^2 + b_1^2 \abs{\nabla a_2}^2 + a_2^2 \abs{\nabla b_1}^2 + b_2^2 \abs{\nabla b_1}^2 + b_1^2 \abs{\nabla b_2}^2 \right) \psi^2(x,y) \\
  &=& \left( \abs{\nabla a_1}^2 + \abs{\nabla b_1}^2 + b_1^2 \left( \abs{\nabla a_2}^2 + \abs{\nabla b_2}^2 \right) \right) \psi^2(x,y) \\
  &\leq & \frac{8k^2}{r^2} \psi^2(x,y) = \frac{24k^2}{\alpha^2} \psi^2(x,y).
\end{IEEEeqnarray*}

\subsection{Construction for $N=3$, $d=3$}

In this case, let the variable be $X = (x,y,z) = (x_1,x_2,x_3;y_1,y_2,y_3;z_1,z_2,z_3)$, and define as before
\[
 a_j(x,y) = a(x_j-y_j) \smallspace b_1(x,y) = b(x_1-y_1)
\]
for $j=2,3$. As in the case $N=2$, define also
\begin{align*}
 g_1(x,y) &= a_1(x,y) + b_1(x,y) a_2(x,y) + b_1(x,y) b_2(x,y) a_3(x,y) \\
 g_\xi(x,y) &= a_1(x,y) + \xi b_1(x,y) a_2(x,y) + \bar{\xi} b_1(x,y) b_2(x,y) a_3(x,y).
\end{align*}

Now come the definition of the wave-function:
\begin{align*}
 \bpsi_{000}(x,y,z) &= 0 \\
 \bpsi_{001}(x,y,z) &= \frac{1}{3} g_1(x,y)\psi(x,y,z) \\
 \bpsi_{010}(x,y,z) &= -\frac{1}{3} g_1(x,z)\psi(x,y,z) \\
 \bpsi_{100}(x,y,z) &= \frac{1}{3} g_1(y,z) \psi(x,y,z) \\
 \bpsi_{110}(x,y,z) &= \frac{\sqrt{2}}{3} g_\xi(x,y) \psi(x,y,z) \\
 \bpsi_{101}(x,y,z) &= -\frac{\sqrt{2}}{3} g_\xi(x,z) \psi(x,y,z) \\
 \bpsi_{011}(x,y,z) &= \frac{\sqrt{2}}{3} g_\xi(y,z) \psi(x,y,z) \\
 \bpsi_{111}(x,y,z) &= 0.
\end{align*}

It is quite easy to see that $\bpsi$ is indeed fermionic. The fact that
\[
 \sum_{s \in S} \abs{\bpsi_s(x,y,z)}^2 = \psi^2(x,y,z)
\]
is proved exactly in the same way as for $N=2$, and also the gradient estimates, considering the couples $\bpsi_{001}$-$\bpsi_{110}$, $\bpsi_{010}$-$\bpsi_{101}$ and $\bpsi_{100}$-$\bpsi_{011}$. Hence we proved the Proposition \ref{fermionic}, with $C=24k^2$ for $k>1$ arbitrary.

\subsection{The case $d=4$}

A very similar construction may be done for $d=4$. In the case $N=2$ one simply choose
\begin{align*}
 g_1(x,y) = \frac{1}{\sqrt{2}} (a_1 + ib_1a_2 + b_1b_2a_3 + ib_1b_2b_3a_4) \\
 g_2(x,y) = \frac{1}{\sqrt{2}} (a_1 + ib_1a_2 - b_1b_2a_3 - ib_1b_2b_3a_4).
\end{align*}
and
\begin{align*}
 \bpsi_{00}(x,y) &= g_1(x,y) \psi(x,y) \\
 \bpsi_{01}(x,y) &= 0 \\
 \bpsi_{10}(x,y) &= 0 \\
 \bpsi_{11}(x,y) &= g_2 (x,y) \psi(x,y). \\
\end{align*}

It is easy to verify that $\abs{g_1(x,y)}^2 + \abs{g_2(x,y)}^2 = a_1^2 + b_1^2a_2^2 + b_1^2b_2^2a_3^2 + b_1^2b_2^2b_3^2a_4^2$, and proceeding as before
\[
 \left( \abs{g_1(x,y)}^2 + \abs{g_2(x,y)}^2 \right) \psi^2(x,y) = \psi^2(x,y).
\]

To estimate the derivatives, note that
\begin{IEEEeqnarray*}{rCl}
 2 \mathfrak{Re} \left( \overline{g_1}\nabla g_1 \right) &= &a_1 \nabla a_1 + a_1 \nabla (b_1b_2a_3) + b_1a_2\nabla (b_1a_2) + b_1a_2 \nabla (b_1b_2b_3a_4) \\
 && {} + b_1b_2a_3 \nabla a_1 + b_1b_2a_3 \nabla b_1b_2a_3 + b_1b_2b_3a_4 \nabla (b_1a_2) \\
 &&{} + b_1b_2b_3a_4 \nabla (b_1b_2b_3a_4) \\
 2 \mathfrak{Re} \left( \overline{g_2}\nabla g_2 \right) &= &a_1 \nabla a_1 - a_1 \nabla (b_1b_2a_3) + b_1a_2\nabla (b_1a_2) - b_1a_2 \nabla (b_1b_2b_3a_4) \\
 && {} - b_1b_2a_3 \nabla a_1 + b_1b_2a_3 \nabla b_1b_2a_3 - b_1b_2b_3a_4 \nabla (b_1a_2) \\
 &&{} + b_1b_2b_3a_4 \nabla (b_1b_2b_3a_4)
\end{IEEEeqnarray*}

This yelds, using Lemma \ref{no-crossed},
\begin{IEEEeqnarray*}{rCl}
 \IEEEeqnarraymulticol{3}{l}{\left( 2 \mathfrak{Re} \left( \overline{g_1}\nabla g_1 \right) + 2 \mathfrak{Re} \left( \overline{g_2}\nabla g_2 \right) \right) \psi} \\
 \quad &= 2[&a_1 \nabla a_1 + b_1a_2\nabla (b_1a_2) + b_1b_2a_3 \nabla (b_1b_2a_3) + b_1b_2b_3a_4 \nabla (b_1b_2b_3a_4)]\psi \\
 &= 0.
\end{IEEEeqnarray*}

Therefore,
\[
 \abs{\nabla \bpsi_{00}(x,y)}^2 + \abs{\nabla \bpsi_{00}(x,y)}^2 = \abs{\nabla \psi}^2 + \left( \abs{\nabla g_1(x,y)}^2 + \abs{\nabla g_2(x,y)}^2 \right) \psi^2
\]
and we conclude with the estimate
\begin{IEEEeqnarray*}{rCll}
  \IEEEeqnarraymulticol{4}{l}{\left( \abs{\nabla g_1}^2 + \abs{\nabla g_2}^2 \right) \psi^2} \\
  \quad &= &\Bigl( &\abs{\nabla a_1}^2 + b_1^2 \abs{\nabla a_2}^2 + a_2^2 \abs{\nabla b_1}^2 + b_2^2 a_3^2 \abs{\nabla b_1}^2 + b_1^2 a_3^2 \abs{\nabla b_2}^2 \\
  &&& {} + b_1^2b_2^2 \abs{\nabla a_3}^2 + b_2^2b_3^2 \abs{\nabla b_1}^2 + b_1^2b_3^2 \abs{\nabla b_2}^2 + b_1^2b_2^2 \abs{\nabla b_3}^2 \Bigr) \psi^2 \\
  &= &\Big[ &\abs{\nabla a_1}^2 + \abs{\nabla b_1}^2 + b_1^2 \left( \abs{\nabla a_2}^2 + \abs{\nabla b_2}^2 \right) \\
  &&& {} + b_1^2 b_2^2 \left( \abs{\nabla a_3}^2 + \abs{\nabla b_3}^2 \right) \Bigr] \psi^2 \\
  &\leq & \IEEEeqnarraymulticol{2}{l}{\frac{12k^2}{\alpha^2} \psi^2 = \frac{36k^2}{\alpha^2} \psi^2},
\end{IEEEeqnarray*}
which shows that in this case $C$ can be chosen $36k^2$ for $k > 1$ arbitrary.

For 3 particles it suffices to repeat the construction of Subsection 4.5 in order to obtain Proposition \ref{fermionic} with $C=36k^2$.


\subsection{$\Gamma-\limsup$-inequality}

Finally we get the $\Gamma$-$\limsup$ inequality, and thus the entire proof, both in the symmetric and the antisymmetric case.

\medskip

\paragraph{\bf Bosonic case}

We can complete the proof of Theorem \ref{bos-gam}:

\begin{proof}{of Th. \ref{bos-gam}}
We proved equicoerciveness in Subsection 3.1 and, for all $P \in \Pi(\rho)$,  in Subsection 3.2 
we already proved the $\Gamma$-$\liminf$ inequality. Thus it is left to find a family of 
bosonic wave-functions $\bpsi_{\hbar}$ such that $P_{\bpsi_{\hbar}} \rightharpoonup P$ and
 \[
  \limsup_{\hbar \to 0} \gra{T_{\hbar}(\bpsi_\hbar) + V_{ee}(\bpsi_\hbar)} \leq \Co_\Sy(P).
 \]
 
 Define $\eps(\hbar) = \sqrt{\hbar}$, and set
 \begin{align*}
  \bpsi_\hbar(x_1, s_1, \dotsc, x_N, s_N) &= \left \{ \begin{array}{ll}
    \psi_{\eps(\hbar)} (x_1, \dotsc, x_N) & \text{if } s_1 = \dotsb = s_N = 0 \\
    0 & \text{otherwise} \end{array}
    \right. \\
    {} &= \begin{pmatrix} \psi_{\eps(\hbar)} (x_1, \dotsc, x_N) \\ 0 \\ \vdots \\ 0 \end{pmatrix},
 \end{align*}
 where $\psi_\eps$ are given by Proposition \ref{new-weak}.
 
 These wave-functions are clearly bosonic, and satisfy $\abs{\bpsi_\hbar(X)}^2 = \psi_{\eps(\hbar)}^2(X)$, $\abs{\nabla \bpsi_\hbar(X)}^2 = \abs{\nabla \psi_{\eps(\hbar)}(X)}^2$, so that
 \begin{align*}
  T_{\hbar}(\bpsi_{\hbar}) &= \frac{\hbar^2}{2} \int \abs{\nabla \psi_{\eps(\hbar)}(X)}^2 dX \leq \frac{N\hbar^2}{2} \left ( \norm{\rho}_{H^2}^2 + \frac{K}{4\eps(\hbar)^2} \right ) \\
  {} &= \frac{N\hbar^2}{2} \norm{\rho}_{H^2}^2 + \frac{KN\hbar}{8},
 \end{align*}
 while
 \[
  V_{ee}(\bpsi_{\hbar}) = \int c(X) \psi_{\eps(\hbar)}^2(X) dX = \Co_\Sy(P_{\eps(\hbar)}).
 \]
 
 Now the thesis follows from Proposition \ref{new-weak}, since
 \[
  \limsup_{\hbar \to 0} \gra{ T_{\hbar}(\bpsi_{\hbar}) + V_{ee}(\boldsymbol{\psi}_{\hbar}) } = \limsup_{\hbar \to 0} V_{ee}(\bpsi_{\hbar}) \leq \Co_\Sy(P).
 \]

\end{proof}

\medskip

\paragraph{\bf Fermionic case} We can complete the proof of Theorem \ref{fer-gam}:

\begin{proof}{of Th. \ref{fer-gam}}
We proved equicoerciveness in Subsection 3.1 and, for all $P \in \Pi(\rho)$, in Subsection 3.2, 
we already proved the $\Gamma$-$\liminf$ inequality. Thus it is left to find a family of 
fermionic wave-functions $\bpsi_{\hbar}$ such that  $P_{\bpsi_{\hbar}} \rightharpoonup P$ and
  \[
  \limsup_{\hbar \to 0} \gra{T_{\hbar}(\bpsi_\hbar) + V_{ee}(\bpsi_\hbar)} \leq \Co_\Sy(P).
 \]
 
 Consider a sequence of functions $\gra{\psi_\eps}$ as in the thesis of Proposition \ref{new-weak}.
 Recall that $\psi_\eps$ is supported outside $D_{\alpha(\eps)}$, where $\alpha(\eps) \searrow 0$ as $\eps \to 0$ -- hence there exists $\alpha^{-1}$ in a right neighbourhood of 0. We may then consider a corresponding family of wave-functions $\gra{\bpsi_\eps}$ given by Proposition \ref{fermionic}. Define
  \[
  \eps(\hbar) = \max \gra{\alpha^{-1}(\sqrt{\hbar}), \sqrt{\hbar}},
  \]
  and observe that $\eps(\hbar) \to 0$ as $\hbar \to 0$. We take $\bpsi_\hbar = \bpsi_{\eps(\hbar)}$ as a recovery sequence.

  It is easy to estimate the kinetic energy:
  \begin{align*}
  T_\hbar(\bpsi_{\eps(\hbar)}) &= \frac{\hbar^2}{2} \int \abs{\nabla \bpsi_{\eps(\hbar)}(X)}^2 dX \\
  {} &\leq \frac{\hbar^2}{2} \gra{\int \abs{\nabla \psi_{\eps(\hbar)}(X)}^2 dX + \frac{C}{\alpha^2(\eps(\hbar))} \int \psi_\eps^2(X) dX} \\
  {} &\leq \frac{\hbar^2}{2} \gra{N \norm{\sqrt{\rho}}^2_{H^1} + \frac{K}{4\eps^2(\hbar)} + \frac{C}{\hbar}} \\
  {} &\leq \frac{\hbar^2}{2} \gra{N \norm{\sqrt{\rho}}^2_{H^1} + \frac{K}{4\hbar} + \frac{C}{\hbar}}
  \end{align*}
  which tends to 0 as $\hbar \to 0$. On the other hand, with the notation of Proposition \ref{new-weak},
  \[
    V_{ee}(\bpsi_\hbar) = \int c(X) \psi_{\eps(\hbar)}^2(X) dX = \Co_\Sy(P_{\eps(\hbar)}).
  \]
 
  Now the thesis follows, since
  \[
   \limsup_{\hbar \to 0} \gra{ T_{\hbar}(\bpsi_\hbar) + V_{ee}(\bpsi_\hbar) } = \limsup_{\hbar \to 0} V_{ee}(\bpsi_\hbar) \leq\Co_\Sy(P).
  \] 
\end{proof}

\subsection{Conclusions}
The $\Gamma-$convergence result of the previous section allow us to prove Theorems \ref{bosonconv} and \ref{fermionconv}:

\begin{proof}{of Th.\ref{bosonconv}.} \ 
Thanks to the formulations (\ref{refbos}) and (\ref{reftrasp}) the thesis follows from Theorem \ref{convminima}, applied to the functionals $\F^{\Sy}_{\hbar}$ and $\Co_\Sy$. 
We proved the $\Gamma-$convergence and equicoercivity in Th. \ref{bos-gam}.
\end{proof}

\begin{proof}{of Th.\ref{fermionconv}}
Thanks to the formulations (\ref{refferm}) and (\ref{reftrasp}) the thesis follows from Theorem \ref{convminima}, applied to the functionals $\F^{\A}_{\hbar}$ and $\Co_\Sy$. 
We proved the $\Gamma-$convergence and equicoercivity in Th. \ref{fer-gam}.
\end{proof}

\section*{Acknowledgement}

The research of the second author is part of the project 2010A2TFX2 {\it Calcolo delle Variazioni} funded by the Italian Ministry of Research and is partially financed by the {\it``Fondi di ricerca di ateneo''} of the University of Pisa and by the INDAM-GNAMPA.

\bibliographystyle{plain}

\bigskip
{\small\noindent
Ugo Bindini:
Scuola Normale Superiore,\\
Piazza dei Cavalieri 5,
56127 Pisa - ITALY\\
{\tt ugo.bindini@sns.it}\\

\bigskip\noindent
Luigi De Pascale:
Dipartimento di Matematica e Informatica,
Universit\`a di Firenze\\
Viale Morgagni 67/A ,
50134 Firenze - ITALY\\
{\tt luigi.depascale@unifi.it}\\
{\tt http://web.math.unifi.it/\textasciitilde depascal/}

\end{document}